\documentclass[11pt,a4paper]{article}

\usepackage{amsmath}
\usepackage{amsthm}
\usepackage{amstext}
\usepackage{enumerate}
\usepackage{textcomp}
\usepackage{amsfonts}
\usepackage{amssymb,enumerate,caption}
\usepackage{graphics}

\newtheorem{theorem}{Theorem}
\newtheorem{lemma}{Lemma}[section]
\newtheorem{definition}{Definition}[section]

\begin{document}

\begin{center}
\Large{Error Estimates for Approximating Best Proximity Points for Cyclic Contractive Maps}
\end{center}

\begin{center}
Boyan Zlatanov
\end{center}

\textbf{Abstract:}
We find a priori and a posteriori error estimates of the best proximity point for the Picard iteration
associated to a cyclic contraction map, which is defined on a uniformly convex Banach space with modulus of convexity of power type.

{\bf Keywords:} best proximity points, uniformly convex Banach space, modulus of convexity, a priori error
estimate, a posteriori error estimate

{\bf AMS Subject Classification:} 41A25, 47H10, 54H25, 46B20

\section{Introduction}

A fundamental result in fixed point theory is the Banach Contraction Principle.
Fixed point theory is an important tool for solving equations $Tx = x$ for mappings $T$
defined on subsets of metric spaces or normed spaces. One of the advantage of Banach fixed point Theorem is the error estimates of the successive iterations and the rate of convergence. There are equations $Tx = x$ for which the exact solution is not easy to find or even
is not possible to find. The error estimate is very useful in these cases. An extensive study about approximations of fixed points can be found in \cite{Ber}.
One kind of a generalization of the Banach Contraction Principle is the notation of cyclical maps \cite{KSV}, i.e. $T(A)\subseteq B$ and $T(B)\subseteq A$.
Because a non-self mapping $T:A\to B$ does not necessarily have a fixed point, one often attempts to find an element
$x$ which is in some sense closest to $Tx$. Best proximity point theorems are relevant in this perspective.
The notation of best proximity point is introduced in \cite{EV}.
This definition is more general than the notation of cyclical maps, in sense that if the sets intersect,
then every best proximity point is a fixed point. A sufficient condition for existence and the uniqueness of best proximity points in uniformly convex Banach spaces is given in \cite{EV}. Since the publication \cite{EV} the problem for existence and uniqueness of best proximity point was widely investigated
see for example \cite{MR,P} and the research on this problem continues.

In contrast with all the results about fixed points for self maps and cyclic maps, where
''a priori error estimates`` and ''a posteriori error estimates`` are obtained
there are no such results about best proximity points.

We have obtained ''a priori error estimates`` and ''a posteriori error estimates`` for the cyclic contractions from \cite{EV}.

\section{Preliminaries}

In this section we give some basic definitions and concepts which are useful and related to the best proximity points.
Let $(X,\rho )$ be a metric space.
Define a distance between two subset $A,B\subset X$ by ${\rm dist}(A,B)=\inf\{\rho (x,y):x\in A, y\in B\}$.
For simplicity of the notations we will denote ${\rm dist}(A,B)$ with $d$.

Let $A$ and $B$ be nonempty subsets of a metric space $(X,\rho )$.
The map $T:A\bigcup B\to A\bigcup B$ is called a cyclic map if $T(A)\subseteq B$ and $T(B)\subseteq A$.
A point $\xi\in A$ is called a best proximity point of the cyclic map $T$ in $A$ if $\rho (\xi,T\xi )={\rm dist}(A,B)$.

Let $A$ and $B$ be nonempty subsets of a metric space $(X,\rho )$.
The map $T:A\bigcup B\to A\bigcup B$ is called a cyclic contraction map if $T$ is a cyclic map and
for some $k\in (0,1)$ there holds the inequality $\rho (Tx,Ty)\leq k\rho (x,y)+(1-k)d$
for any $x\in A$, $y\in B$. The definition for cyclic contraction is introduced in \cite{EV}.

The best proximity results need norm-structure of the space $X$.
When we investigate a Banach space $(X,\|\cdot\| )$ we will always consider the
distance between the elements to be generated by the norm $\|\cdot\|$ i.e. $\rho (x,y)=\|x-y\|$. We will denote the unit sphere
and the unit ball of a Banach space $(X,\|\cdot\| )$ by $S_X$ and $B_X$ respectively.

The assumption that the Banach space $(X,\|\cdot\|)$ is uniformly convex
plays a crucial role in the investigation of best proximity points.

\begin{definition}
Let $(X,\|\cdot\|)$ be a Banach space. For every $\varepsilon\in (0,2]$ we define the modulus of convexity of $\|\cdot\|$ by
$$
\delta_{\|\cdot\|}(\varepsilon)=\inf\left\{1-\left\|\frac{x+y}{2}\right\|:x,y\in B_X, \|x-y\|\geq\varepsilon\right\}.
$$
The norm is called uniformly convex if $\delta_X(\varepsilon)>0$ for all $\varepsilon\in (0,2]$. The space $(X,\|\cdot\|)$ is then called uniformly convex space.
\end{definition}

The results from \cite{EV} and \cite{KA} are summarized in the next theorem.
\begin{theorem}\label{EV}(\cite{EV,KA})
Let $A$ and $B$ be nonempty closed and convex subsets of a uniformly convex Banach space. Let $T:A\cup B\to A\cup B$ be a cyclic contraction map. Then there is a unique best proximity point $\xi$ of $T$ in $A$, $T\xi$ is a unique best proximity point of $T$ in $B$ and $\xi =T^2\xi =T^{2n}\xi$. Further if $x_0\in A$ and $x_{n+1}=Tx_n$, then $\{x_{2n}\}_{n=1}^\infty$ converges to $\xi$ and $x_{2n+1}$ converges to $T\xi$.
\end{theorem}

For any uniformly convex Banach space $X$ there holds the inequality
\begin{equation}\label{eq:1}
\left\|\frac{x+y}{2}-z\right\|\leq \left(1-\delta_X\left(\frac{r}{R}\right)\right)R
\end{equation}
for any $x,y,z\in X$, $R>0$, $r\in [0,2R]$, $\|x-z\|\leq R$, $\|y-z\|\leq R$ and $\|x-y\|\geq r$.

If $(X,\|\cdot\|)$ is a uniformly convex Banach space, then $\delta_X(\varepsilon)$ is strictly increasing function.
Therefore if $(X,\|\cdot\|)$ is a uniformly convex Banach space then there exists the inverse function $\delta^{-1}$ of the modulus of convexity.
If there exist constants $C>0$ and $q>0$, such that the inequality $\delta_{\|\cdot\|}(\varepsilon)\geq C\varepsilon^q$ holds for every $\varepsilon\in (0,2]$ we say that the modulus of convexity is of power type $q$. It is well known that
 for any Banach space and for any norm there holds the inequality
$\delta (\varepsilon)\leq K\varepsilon^2$. The modulus of convexity with respect to the canonical norm $\|\cdot\|_p$ in $\ell_p$ or $L_p$ is
$\delta_{\|\cdot\|_p}(\varepsilon)=1-\root p \of {1-\left(\frac{\varepsilon}{2}\right)^p}$ for $p\geq 2$
and for $1<p<2$ the modulus of convexity $\delta_{\|\cdot\|_p} (\varepsilon)$ is the solution of the equation
$\left(1-\delta+\frac{\varepsilon}{2}\right)^p+\left|1-\delta -\frac{\varepsilon}{2}\right|^p=2$.
It is well known that the modulus of convexity with respect to the canonical norm in $\ell_p$ or $L_p$ is of power type and there
 holds the inequalities $\delta_{\|\cdot\|_p}(\varepsilon)\geq \frac{\varepsilon^p}{p2^p}$ for $p\geq 2$ and
$\delta_{\|\cdot\|_p}(\varepsilon)\geq \frac{(p-1)\varepsilon^2}{8}$ for $p\in (1,2)$ \cite{Meir}.

An extensive study of the Geometry of Banach spaces can be found in \cite{BB, DGZ, FHHMPZ}.
The next lemma is easy to get and it is used without stating it in most of the articles about best proximity points.
\begin{lemma}\label{lem:1}
Let $A$ and $B$ be nonempty subsets of a metric space $(X,\rho )$ and let $T:A\cup B\to A\cup B$ be a cyclic contraction map.
Then for every $x\in A\cup B$ there holds the inequality
$\rho (T^nx,T^{n+1}x)-d\leq k^n\left(\rho (x,Tx)-d\right)$.
\end{lemma}

\section{Error estimates for best proximity points}

\begin{theorem}\label{thm:main}
Let $A$ and $B$ be nonempty, closed and convex subsets of a uniformly convex Banach $(X,\|\cdot\|)$ space, such that $d={\rm dist}(A,B)>0$, and let there exist $C>0$ and $q\geq 2$, such that
$\delta_{\|\cdot\|}(\varepsilon)\geq C\varepsilon^q$. Let $T:A\cup B\to A\cup B$ be a cyclic contraction map. Then
\begin{enumerate}[\normalfont (i)]
\item there exists a unique best proximity point $\xi$ of $T$ in $A$,
$T\xi$ is a unique best proximity point of $T$ in $B$ and $\xi =T^2\xi =T^{2n}\xi$;
\item for any $x_0\in A$ the sequence $\{x_{2n}\}_{n=1}^\infty$ converges to $\xi$ and $\{x_{2n+1}\}_{n=1}^\infty$ converges to $T\xi$,
where $x_{n+1}=Tx_n$, $n=0,1,2,\dots $;
\item a priori error estimate holds
\begin{equation}\label{eq:2}
\left\|\xi-T^{2n}x\right\|
\leq\frac{\|x-Tx\|}{1-\root {q} \of {k^2}}\root{q}\of{\frac{\|x-Tx\|-d}{Cd}}
\left(\root {q} \of k\right)^{2n};
\end{equation}
\item a posteriori error estimate holds
\begin{equation}\label{eq:3}
\left\|T^{2n}x-\xi\right\|\leq
\frac{\|T^{2n-1}x-T^{2n}x\|}{1-\root {q} \of {k^2}}\root{q}\of{\frac{\|T^{2n-1}x-T^{2n}x\|-d}{Cd}}
\root {q} \of k.
\end{equation}
\end{enumerate}
\end{theorem}

\begin{proof}
The proof of (i) and (ii) follows from Theorem \ref{EV}.

We will use the notation $S_{n,m}(x)=\|T^{n}x-T^{m}x\|-d$, just to be able to fit some of the formulas in the text field.

(iii) For any $x\in A$, $n\in\mathbb{N}$ and $l\leq 2n$ there holds the inequality
$$
\delta_{\|\cdot\|}\left(\frac{\|T^{2n}x-T^{2n+2}x\|}{d+k^{l}S_{2n-l,2n+1-l}(x)}\right)
\leq\frac{k^{l}S_{2n-l,2n+1-l}(x)}{d+k^{l}S_{2n-l,2n+1-l}(x)}.
$$

Indeed let $x\in A$ be arbitrary chosen.
From Lemma \ref{lem:1} we have the inequalities
$$
\|T^{2n}x-T^{2n+1}x\|\leq d+k^{l}S_{2n-l,2n+1-l}(x),
$$
$$
\|T^{2n+2}x-T^{2n+1}x\|\leq d+k^{l+1}S_{2n-l,2n+1-l}(x)<d+k^{l}S_{2n-l,2n+1-l}(x)
$$
and
$$
\begin{array}{lll}
\|T^{2n+2}x-T^{2n}x\|&\leq&
 \|T^{2n+2}x-T^{2n+1}x\|+\|T^{2n+1}x-T^{2n}x\|\\
 &\leq& 2\left(d+k^{l}S_{2n-l,2n+1-l}(x)\right).
\end{array}
$$
After a substitution in (\ref{eq:1}) with $x=T^{2n}x$, $y=T^{2n+2}x$, $z=T^{2n+1}x$,
$r=\|T^{2n+2}x-T^{2n}x\|$ and $R=d+k^{l}\left(\|T^{2n-l}x-T^{2n+1-l}x\|-d\right)=d+k^{l}S_{2n-l,2n+1-l}(x)$
and using the convexity of the set $A$ we get the chain of inequalities
\begin{equation}\label{eq:5}
\begin{array}{lll}
d&\leq&\left\|\frac{T^{2n}x+T^{2n+2}x}{2}-T^{2n+1}x\right\|\\[12pt]
&\leq&\left(1-\delta_{\|\cdot\|}\left(\frac{\|T^{2n}x-T^{2n+2}x\|}{d+k^{l}S_{2n-l,2n+1-l}(x)}\right)\right)\left(d+k^{l}S_{2n-l,2n+1-l}(x)\right).
\end{array}
\end{equation}
From (\ref{eq:5}) we obtain the inequality
\begin{equation}\label{eq:61}
\delta_{\|\cdot\|}\left(\frac{\|T^{2n}x-T^{2n+2}x\|}{d+k^{l}S_{2n-l,2n+1-l}(x)}\right)
\leq\frac{k^{l}S_{2n-l,2n+1-l}(x)}{d+k^{l}S_{2n-l,2n+1-l}(x)}.
\end{equation}
From the uniform convexity of $X$ is follows that $\delta_{\|\cdot\|}$ is strictly increasing and therefore there exists its inverse function
$\delta^{-1}_{\|\cdot\|}$, which is strictly increasing too. From (\ref{eq:61}) we get
\begin{equation}\label{eq:6}
\|T^{2n}x-T^{2n+2}x\|
\leq \left(d+k^{l}S_{2n-l,2n+1-l}(x)\right)\delta^{-1}_{\|\cdot\|}\left(\frac{k^{l}S_{2n-l,2n+1-l}(x)}{d+k^{l}S_{2n-l,2n+1-l}(x)}\right).
\end{equation}
By the inequality $\delta_{\|\cdot\|}(t)\geq Ct^q$ it follows that $\delta^{-1}_{\|\cdot\|}(t)\leq \left(\frac{t}{C}\right)^{1/q}$. From
(\ref{eq:6}) and the inequalities $d\leq d+k^{l}S_{2n-l,2n+1-l}(x)\leq \|T^{2n-l}x-T^{2n+1-l}x\|$ we obtain
\begin{equation}\label{eq:7}
\begin{array}{lll}
\|T^{2n}x-T^{2n+2}x\|&\leq&
\left(d+k^{l}S_{2n-l,2n+1-l}(x)\right)\root{q}\of{\frac{k^{l}S_{2n-l,2n+1-l}(x)}
{C.\left(d+k^{l}S_{2n-l,2n+1-l}(x)\right)}}\\
&\leq&
\|T^{2n-l}x-T^{2n+1-l}x\|\root{q}\of{\frac{S_{2n-l,2n+1-l}(x)}{Cd}}\left(\root{q}\of{k}\right)^l.
\end{array}
\end{equation}

From (i) and (ii) there exists a unique $\xi$, such that $\|\xi-T\xi\|=d$, $T^2\xi=\xi$
and $\xi$ is a limit of the sequence $\{T^{2n}x\}_{n=1}^\infty$ for any $x\in A$.

After a substitution with $l=2n$ in (\ref{eq:7}) we get the inequality
$$
\begin{array}{lll}
\sum_{n=1}^\infty\left\|T^{2n}x-T^{2n+2}x\right\|&\leq&
\|x-Tx\|\root{q}\of{\frac{\|x-Tx\|-d}{Cd}}\sum_{n=1}^\infty \left(\root {q} \of k\right)^{2n}\\
&=&\|x-Tx\|\root{q}\of{\frac{\|x-Tx\|-d}{Cd}}\cdot\frac{\root {q} \of {k^2}}{1-\root {q} \of {k^2}}
\end{array}
$$
and consequently the series $\sum_{n=1}^\infty (T^{2n}x-T^{2n+2}x)$ is absolutely convergent.
Thus for any $m\in\mathbb{N}$ there holds
$\xi=T^{2m}x-\sum_{n=m}^\infty \left(T^{2n}x-T^{2n+2}x\right)$
and therefore we get the inequality
$$
\left\|\xi-T^{2m}x\right\|\leq \sum_{n=m}^\infty \left\|T^{2n}x-T^{2n+2}x\right\|
\leq\|x-Tx\|\root{q}\of{\frac{\|x-Tx\|-d}{Cd}}\cdot\frac{\left(\root {q} \of k\right)^{2m}}{1-\root {q} \of {k^2}}.
$$

(iv) We will use the notation $P_{n,m}(x)=\|T^{n}x-T^{m}x\|$, just to be able to fit some of the formulas in the text field.
After a substitution with $l=1+2i$ in (\ref{eq:7}) we obtain
\begin{equation}\label{eq:9}
P_{2n+2i,2n+2(i+1)}(x)\leq
P_{2n-1,2n}(x)\root{q}\of{\frac{P_{2n-1,2n}(x)-d}{Cd}}\left(\root{q}\of{k}\right)^{1+2i}.
\end{equation}
From (\ref{eq:9}) we get that there holds the inequality
\begin{equation}\label{eq:10}
\begin{array}{lll}
P_{2n,2(n+m)}(x)&\leq&
\sum_{i=0}^{m-1}P_{2n+2i,2n+2(i+1)}(x)\\
&\leq&\sum_{i=0}^{m-1}P_{2n-1,2n}(x)\root{q}\of{\frac{P_{2n-1,2n}(x)-d}{Cd}}\left(\root{q}\of{k}\right)^{1+2i}\\
&=&P_{2n-1,2n}(x)\root{q}\of{\frac{P_{2n-1,2n}(x)-d}{Cd}}\sum_{i=0}^{m-1} \left(\root {q} \of k\right)^{1+2i}\\
&=&P_{2n-1,2n}(x)\root{q}\of{\frac{P_{2n-1,2n}(x)-d}{Cd}}\cdot
\frac{1-\left(\root {q} \of k\right)^{2m}}{1-\root {q} \of {k^2}}\root {q} \of k
\end{array}
\end{equation}
and after letting $m\to\infty$ in (\ref{eq:10}) we obtain the inequality
$$
\left\|T^{2n}x-\xi\right\|\leq
\|T^{2n-1}x-T^{2n}x\|\root{q}\of{\frac{\|T^{2n-1}x-T^{2n}x\|-d}{Cd}}
\frac{\root {q} \of k}{1-\root {q} \of {k^2}}.
$$
\end{proof}
\section{Remarks and an Example}

Following \cite{Ber} we would like to say a few words about the error estimates.

The a priori estimate (\ref{eq:2}) shows that, when starting from an initial guess
$x\in A$ the upper bound of approximation error for the $2n$ iterate is completely determined
by the cyclic contraction coefficient $k$ and the initial displacement $\|x-Tx\|$.

Similarly, the a posteriori estimate shows that, in order to obtain the
desired error approximation $\|T^{2n}-\xi\|<\varepsilon$ of the fixed point by means of Picard iteration
we need to stop the iterative process at the first
step $2n$ for which the displacement between two consecutive iterates satisfies the inequality
$\frac{\|T^{2n-1}x-T^{2n}x\|}{1-\root {q} \of {k^2}}\root{q}\of{\frac{\|T^{2n-1}x-T^{2n}x\|-d}{Cd}}
\root {q} \of k<\varepsilon$.
Thus the a posteriori estimation offers a direct stopping criterion for the
iterative approximation of fixed points by Picard iteration, while the a priori
estimation indirectly gives a stopping criterion.

We will illustrate Theorem \ref{thm:main} with the next example.

{\bf Example 1:} Let consider the space $\mathbb{R}^2=\{(x,y):x,y\in\mathbb{R}\}$ endowed with the norms
$\|x\|_p=\root{p}\of{|x|^p+|y|^p}$, for $p>1$. The space $(\mathbb{R},\|\cdot\|_p)$ is uniformly convex with modulus of convexity
of power type, provided that $p>1$.
Let us consider the sets
$A=\{(x,y)\in\mathbb{R}^2:y-x+1\leq 0, y+x-1\geq 0\}$
and
$B=\{(x,y)\in\mathbb{R}^2:y-x-1\geq 0, y+x+1\leq 0\}$.

It is easy to calculate ${\rm dist}(A,B)=2$.
Let $\lambda \in(0,1)$ .
Let us define a map $T:\mathbb{R}^2_p\to\mathbb{R}^2_p$ by
$T(x,y)=\left(-((1-\lambda ){\rm sign} (x)+\lambda x),-\lambda y\right)$,
where ${\rm sign} (x)=1$ if $x>0$, ${\rm sign} (x)=-1$ if $x<0$ and ${\rm sign} (x)=0$ if $x=0$.

We will show that the map $T:A\cup B\to A\cup B$ is a cyclic contraction with $k=\lambda$.

Let $z=(x,y)\in A$. From $x,y\geq 0$ we get
$-\lambda y-(1-\lambda +\lambda x)+1=-(\lambda y+\lambda x-\lambda )\leq 0$ and
$-\lambda y+(1-\lambda +\lambda x)-1=-(\lambda y-\lambda x+\lambda )\geq 0$.
Therefore $T(A)\subseteq B$. The inclusion $T(B)\subseteq A$ is proven in a similar fashion.

Let us put $u_1=(x_1,y_1)\in A$, $u_2=(x_2,y_2)\in B$ and $e_1=(1,0)\in A$. It is easy to observe that $e_1$ is a best proximity point of $T$ in
$A$, $T(e_1)=-e_1$ and $T^2(e_1)=T(-e_1)=e_1$.
We get the chain of inequalities
$$
\begin{array}{lll}
\|T(x_1,y_1)-T(x_2,y_2)\|_p&\hspace{-6pt}\leq&\hspace{-6pt}
\root{p}\of{\left|2(1-\lambda)+\lambda(x_1+|x_2|)\right|^p+\left|\lambda(y_1+|y_2|)\right|^p}\\
&\hspace{-6pt}\leq&\hspace{-6pt}\left\|2(1-\lambda)e_1+\lambda(u_1-u_2)\right\|_p\\
&\hspace{-6pt}\leq&\hspace{-6pt}\lambda\|u_1-u_2\|_p+2(1-\lambda)\|e_1\|_p\\
&\hspace{-6pt}\leq&\hspace{-6pt} \lambda\|u_1-u_2\|_p+(1-\lambda)d.
\end{array}
$$
Thus we can apply Theorem \ref{thm:main} to get error estimates of the successive iterations $\{x_{2n}\}_{n=1}^\infty$, where $x_{n+1}=Tx_n$.

We will consider a numeric example with $\lambda=2^{-1}$.
From \cite{Meir} we get $C=\displaystyle\frac{1}{p2^p}$, $q=p$ for $p\geq 2$ and $C=\displaystyle\frac{p-1}{8}$, $q=2$ for $p\in (1,2]$.

\begin{center}
\captionof{table}{Number $2n$ of iterations, needed by the a posteriori estimate for $\lambda=2^{-1}$
with an initial point $x_0 =(1000,8)$}
\label{table1}
\begin{tabular*}{0.9\textwidth}{@{\extracolsep{\fill} }lrrrrrr}
\hline
$\varepsilon\setminus p$&1.1&1.5&2&3&5&20\\
\hline
$10^{-2} $&34&32&30&42 &66 &266\\
$10^{-4} $&48&46&44&62 &100&398\\
$10^{-6} $&60&58&58&82 &132&532\\
$10^{-8} $&74&72&70&102&166&664\\
$10^{-10}$&88&84&84&122&200&798\\
\hline
\end{tabular*}
\end{center}

\begin{center}
\captionof{table}{Number $2n$ of iterations, needed by the a priori estimate for $\lambda=2^{-1}$
with an initial point $x_0 =(1000,8)$}\label{table2}
\begin{tabular*}{0.9\textwidth}{@{\extracolsep{\fill} }lrrrrrr}
\hline
$\varepsilon\setminus p$&1.1&1.5&2&3&5&20\\
\hline
$10^{-2} $&54&50&46&64 &104 &428\\
$10^{-4} $&66&64&58&84 &138&560\\
$10^{-6} $&80&78&72&104 &170&694\\
$10^{-8} $&94&90&86&124&204&826\\
$10^{-10}$&106&104&98&144&238&960\\
\hline
\end{tabular*}
\end{center}

\section{Conclusion and open questions}

We would like to mention that the error estimates give much larger number of the iterations that are needed.
It is due to the fact that we use the modulus of convexity, which is the infinum of $1-\left\|\frac{x+y}{2}\right\|$
among all $x,y\in S_x$, such that $\|x-y\|\geq\varepsilon$. It may happen that the modulus of convexity is greater in the direction of the best proximity point $\xi$ than in the other directions but for the estimation of the error we do not use it. We would like to pose the following question is it possible to get better estimates if we use the directional modulus of convexity $\delta_{\|\cdot\|}(x,\varepsilon)$?

For the estimations we use geometric progression and that is why we impose the condition for the modulus of convexity to be of power type. Is it possible to obtain error estimates if the modulus of convexity is not of power type?

Is it possible to obtain error estimates for the sequence of successive iterates for weak cyclic Kannan contractions \cite{P} and
for cyclic $\phi$--contractions \cite{MR}?

\bibliographystyle{amsplain}

\end{document}